\documentclass[12 pt,leqno,english]{smfart}
\usepackage{euscript}
\usepackage{amssymb}
\usepackage{amsfonts}
\usepackage{smfthm}
\usepackage{smfenum}
\usepackage{graphicx}
\usepackage[english]{babel}
\usepackage{subfig}
\usepackage{url}

\newcommand{\C}{{\mathbb{C}}}
\newcommand{\Z}{{\mathbb{Z}}}
\newcommand{\Q}{{\mathbb{Q}}}
\newcommand{\N}{{\mathbb{N}}}

\newcommand{\Nm}{{\mathrm{Nm}}}
\newcommand{\OO}{{\EuScript{O}}}

\author{Hester Graves}
\email{gravesh@umich.edu}
\author{Nick Ramsey}
\email{naramsey@gmail.com}
\title{Euclidean Ideals in Quadratic Imaginary Fields}
\begin{document}
\frontmatter

\begin{abstract}
  We classify all quadratic imaginary number fields that have a
  Euclidean ideal class.  There are seven of them, they are of class
  number at most two, and in each case the unique class that generates
  the class-group is moreover norm-Euclidean.
\end{abstract}

\maketitle

\section{Introduction}

In \cite{lenstra}, Lenstra generalized the notion of a Euclidean
domain to that of a Dedekind domain $R$ with a Euclidean ideal
$C\subseteq R$.  He proved that if $C$ is a Euclidean ideal then the
class-group of $R$ is cyclic and generated by $C$.  Moreover, if
$C=R$, then his notion reduces to that of a Euclidean domain and the
above result reduces to the familiar fact that a Euclidean domain is a
principal ideal domain.  Building on work of Weinberger
(\cite{weinberger}), Lenstra (\cite{lenstra}) showed (conditional of
the generalized Riemann hypothesis) that any generator of the class
group of the ring of integers in a number field with infinite unit
group is Euclidean.   As the only
number fields with finite unit group aside from $\Q$, it is natural to
inquire about the situation for quadratic imaginary fields.

It is known that, among the nine quadratic imaginary fields of class
number one, exactly five have Euclidean integer rings and in each case
the norm serves as a Euclidean algorithm (see \cite{samuel}).  The purpose of
this paper is to extend this result to the setting of Euclidean ideal
classes by determining all quadratic imaginary fields that have a
Euclidean ideal.  We record them in the following theorem.
\begin{theo}\label{thm1}
  The quadratic imaginary fields with a Euclidean ideal are as follows.
\begin{center}
  \begin{tabular}{c|c} {\rm class number} & {\rm fields}\\\hline
     1 & $\Q(\sqrt{-D})$ for $D\in\{1, 2, 3, 7, 11\}$ \\
     2 & $\Q(\sqrt{-D})$ for $D\in\{5,15\}$
  \end{tabular}
\end{center}
In each case the unique class that generates the class-group is
moreover norm-Euclidean.
\end{theo}
If one is only interested in \emph{norm}-Euclidean ideals, then this
result is contained in Proposition 2.4 of \cite{lenstra}.  Of course,
as the results of Weinberger and Lenstra mentioned above conditionally
demonstrate and examples of Clark (\cite{clark}) and Harper
(\cite{harper}) unconditionally demonstrate (in the class number one
case), there are ideals in integer rings that are are Euclidean but
not with respect to the norm.

In the Euclidean ring setting, a construction of Motzkin
(\cite{motzkin}) has proven to be a fruitful tool in the study of
Euclidean rings that are not norm-Euclidean.  In her thesis, the first
author adapts this construction to the Euclidean ideal setting.  Her
techniques are the main tool used to prove Theorem \ref{thm1}.

\mbox{}\\

\noindent{\bf Convention} -  In this paper all Euclidean algorithms
are taken to be $\N$-valued and $\N$ is taken to include $0$.

\section{A Motzkin-type construction for Euclidean ideals}

Let $R$ be a Dedekind domain with fraction field $K$.  We denote by
$E$ the set of fractional ideals of $R$ in $K$ that contain $R$ itself.
Recall from \cite{lenstra} that $C$ is called \emph{Euclidean} if
there exists a function $\psi:E\longrightarrow \N$ such that for all
$I\in E$ and all $x\in IC\setminus C$, there exists $y\in C$ such
that $$\psi((x+y)^{-1}IC)<\psi(I)$$ In this case $\psi$ is called a
\emph{Euclidean algorithm} for $C$.  If $C$ is Euclidean then it
generates the class-group of $R$.  Also, if $\psi$ is a Euclidean
algorithm for $C$ then it is also a Euclidean algorithm for any ideal
in the same class as $C$ and no ideal in a different class than $C$.
These facts are all elementary and can be found in \cite{lenstra}.

The following definition, given in \cite{gravesthesis}, is an
adaptation of Motzkin's construction to the Euclidean ideal setting.
\begin{defi}\label{hesterconstruction}
  Let $C$ be a non-zero ideal in $R$.  We define a nested sequence of
  subsets of $E$ as follows.  Set $A_{C,0} = \{R\}$ and for $i>0$ we
  set $$A_{C,i} = A_{C,i-1}\cup \left\{ I\in E \left| \begin{array}{c}
    \forall x\in IC\setminus C\ \ \exists y\in C
    \\ \mbox{such\ that}\ (x-y)^{-1}IC \in
    A_{C,i-1}\end{array}\right.\right\}$$ Finally, set $A_C = \cup_i
  A_{C,i}$.
\end{defi}
When the ideal $C$ is fixed or otherwise clear from the context, we
will often omit it from the notation and simply use $A_i$ and
$A=\cup_i A_i$.  The significance of this construction is the
following lemma of the first author (see \cite{gravesthesis}).
\begin{lemm}\label{lem1}
  The ideal $C$ is Euclidean if and only if $A=E$.
\end{lemm}
In fact, one can say more.  Namely, if $A=E$, then the function $\psi:
E\longrightarrow \N$ defined by $\psi(I) = i$ if $I\in A_i\setminus
A_{i-1}$ is a Euclidean algorithm for $C$ and is minimal with respect
to this property.

The following two lemmas furnish constraints on the sets $A_{C,i}$ that
will be useful in what follows.  The first is general in nature and
highlights the role of cyclicity of the class-group.  
\begin{lemm}\label{lem2}
  If $I\in A_i\setminus A_{i-1}$ then $[I] = [C^{-i}]$.
\end{lemm}
\begin{proof}
  This is an immediate inductive consequence of Definition
  \ref{hesterconstruction}.
\end{proof}

By definition, any $I\in A_i\setminus A_{i-1}$ has the property that
for all $x\in IC\setminus C$ there exists $y\in C$ such that
$(x+y)^{-1}IC\in A_{i-1}$.  However, using the previous lemma and
ideal class considerations, one can often cut down the set
``$A_{i-1}$'' in this statement.  When $R^\times$ is finite, this
observation is particularly useful because one can use it to
efficiently bound the norm of a new element of $A_i$, as the following
lemma demonstrates.
\begin{lemm}\label{lem3}
  Suppose that $R^\times$ is finite, and suppose that $S\subseteq
  A_{i-1}$ is a subset with the property that, if $I\in A_i\setminus
  A_{i-1}$ then for all $x\in IC\setminus C$ there
  exists $y\in C$ such that $$(x-y)^{-1}IC\in S$$ Then all $I\in
  A_i\setminus A_{i-1}$ have the property that $$\Nm(I^{-1})\leq
  |R^\times||S|+1.$$
\end{lemm}
\begin{proof}
  For $x\in IC\setminus C$, the condition that there exists $y\in C$
  such that $(x-y)^{-1}IC$ is a particular ideal depends only on the
  class of $x$ in $IC/C$.  Fix an ideal $I\in A_i\setminus A_{i-1}$.
  For each non-zero class in $IC/C$ choose a
  representative $x\in IC\setminus C$ and a $y\in C$
  such that $(x-y)^{-1}IC\in S\subseteq A_{i-1}$.  This
  collection of choices amounts to a (decidedly non-canonical)
  function $$(IC\setminus C)/C\longrightarrow S$$ Suppose that two classes in
  $(IC\setminus C)/C$ map to the same ideal and let
  $x_1$ and $x_2$ be their chosen representatives.  Then there exist
  $y_1,y_2\in C$ such that $$(x_1-y_1)^{-1}IC =
  (x_2-y_2)^{-1}IC.$$ It follows that there exists a
  unit $u\in \OO_K^\times$ such that $x_1-y_1=u(x_2-y_2)$, and hence
  $x_1-ux_2 = y_1-uy_2\in C$.  That is, the classes of $x_1$ and $x_2$
  in $IC/C$ differ (multiplicatively) by a unit.  The
  upshot is that the set of nonzero classes in $IC/C$
  modulo the multiplicative action of $R^\times$ injects into
  $S$, and hence $$|IC/C| -1 \leq
  |R^\times||S|$$ But since $R$ is Dedekind, the left
  side is simply $\Nm(I^{-1})-1$, so this is the desired inequality.
\end{proof}

\section{Application to quadratic imaginary fields}

For the remainder of the paper, $K$ will denote a quadratic imaginary
field and $\OO_K$ its ring of integers.  One approach to classifying
the Euclidean $\OO_K$ is to break into cases according to the
factorizations of small rational primes in $\OO_K$ and use Lemmas
\ref{lem2} and \ref{lem3} of the previous section to glean
consequences about the sets $A_i$.  If one uses the crutch of known
lists of quadratic imaginary fields of small class number, then this
approach \emph{nearly} yields Theorem \ref{thm1}.  Indeed, aside from
the known norm-Euclidean cases detailed in this theorem, one finds in
nearly all cases that the sequence of sets $A_i$ stabilizes very
quickly (one needn't ever consider ideals with prime factors of norm
larger than 7).  The one vexing exception is the field
$K=\Q(\sqrt{-23})$.  A bit of computation with SAGE (\cite{sage})
reveals that the $A_i$ in this case contain at least the inverses of
every ideal of norm up to $47$.  Lacking the patience to continue this
computation to its end (and indeed the confidence that it had one), we
decided to switch perspective.

It is convenient to first dispense with the cases where $\OO_K^\times$
is unusually large, namely $K=\Q(\sqrt{-1})$ and $K=\Q(\sqrt{-3})$.
These two fields are well-known to have norm-Euclidean rings of
integers, and for any other $K$ we have $\OO_K^\times = \{\pm 1\}$.
From this point on we assume that $K$ is among the latter fields.  It
then follows from Lemma \ref{lem3} that any $I\in A_1\setminus A_0$
has $\Nm(I^{-1})\leq 3$.  As a result, by Lemma \ref{lem2}, a
Euclidean ideal class in $K$ is represented by a
residue degree one prime lying over $2$ or $3$.

Fix an embedding of $K$ into $\C$.  We will freely identify $K$ with
its image in $\C$ in what follows.  Under this embedding, the field
norm corresponds to the square of the complex absolute value.  Note
that a nonzero fractional ideal $C$ of $K$ is identified with a
lattice in $\C$.  Consider the union of the open disks of radius
$\sqrt{\Nm(C)}$ centered about these lattice points.  It is a simple
consequence of the definition and the above comments that $C$ is
norm-Euclidean if these disks cover all of $\C$ (see also
\cite{lenstra}).  The moral of the following result is that, if this
covering fails too badly, then $C$ cannot possibly be Euclidean for
\emph{any} choice of algorithm.

\begin{prop}\label{prop1}
  Let $K$ and $C$ be as above, and let $U$ denote the union of the
  open disks of radius $\sqrt{\Nm(C)}$ centered at the elements of $C$.  If
  the complement of $U$ in $\C$ contains a nonempty open set, then $C$
  is not a Euclidean ideal.
\end{prop}

Before proceeding with the proof, we need the following lemma,
which effectively states that inverses of fractional ideals of
increasingly large norm are increasingly dense in $K$.
\begin{lemm}\label{lemdense}
  Let $K$ be a quadratic imaginary field and let  $\varepsilon>0$ be any
  positive real number.  There exists a number $M$ such that, for all
  $z\in K$ and all fractional ideals $I$ with $\Nm(I)>M$, there exists
  an element $x\in I^{-1}$ such that $\Nm(x-z)<\epsilon$.
\end{lemm}
\begin{proof}
  Let $I_1, I_2, \dots, I_h$ be a set of representatives of the ideal
  class group of $K$.  Viewing each fractional ideal $I_i^{-1}$ as a
  lattice in $\C$, we see that disks of sufficiently large radius
  centered at the elements of $C$ will cover $\C$.  Thus, for each $i$
  there exists a positive number $M_i$ such that, for each $z'\in K$
  there exists $x'\in I_i^{-1}$ such that $$\Nm(x'-z')=|x'-z'|^2<M_i$$
  Now choose $M$ so that $M>\max_i (M_i\Nm(I_i)/\varepsilon)$.  Let
  $z\in K$ and let $I$ be a fractional ideal with $\Nm(I)>M$.  Choose
  $i$ so that $I=gI_i$ for some $g\in K^\times$ and pick $x'\in
  I_i^{-1}$ such that $\Nm(x'-gz)< M_i$.  Then
\begin{eqnarray*}
  \Nm(g^{-1}x'-z) &= & \Nm(g^{-1})\Nm(x'-gz) \\ &<&
  \frac{\Nm(I_i)}{\Nm(I)}M_i <\varepsilon
\end{eqnarray*}
so that $x=g^{-1}x'\in (I_iI^{-1})I_i^{-1} = I^{-1}$ is the desired
element.
\end{proof}

\begin{proof}(of Proposition \ref{prop1})
  Suppose that the complement of $U$ in $\C$ contains a nonempty open
  set. Arguing by contradiction, let us suppose that $C$ is Euclidean
  for the algorithm $\psi:E\longrightarrow \N$, so by Lemma
  \ref{lem1}, $A=\cup A_i = E$.  Since $K$ is dense in $\C$ under its
  embedding, the complement of $U$ contains an
  $\sqrt{\varepsilon}$-neighborhood of an element $z\in K$ for some
  $\varepsilon>0$.  Let $M$ be as in Lemma \ref{lemdense} for this $K$
  and $\varepsilon$.

  Suppose that $I_0\in E$ and $\Nm(I_0^{-1}C^{-1})>M$.  By Lemma
  \ref{lemdense}, there exists $x\in I_0C$ such that $$|x-z| =
  (\Nm(x-z))^{1/2}<\sqrt{\varepsilon}$$ It follows that $x$ lies in
  the complement of $U$.  Since $x\in I_0C\setminus C$ and $I_0\in
  E=A$, there exists $y\in C$ such that
  $\psi((x+y)^{-1}I_0C)<\psi(I_0)$.  Define $I_1 = (x+y)^{-1}I_0C$ and
  note that $I_1\in E$ and $\psi(I_1)<\psi(I_0)$.  By the above,
  we also have $$\Nm(I_1) = \Nm(I_0)\Nm(C)/\Nm(x+y) =
  \Nm(I_0)\Nm(C)/|x+y|^2 \leq \Nm(I_0)$$ since $x$ lies in the
  complement of $U$.

  Because of this norm inequality, we again have $\Nm(I_1^{-1}/C)>M$
  and can repeat the argument with $I_0$ replaced by $I_1$ to obtain a
  fractional ideal $I_2\in E$ with $\psi(I_2)<\psi(I_1)$ and
  $\Nm(I_2)\leq \Nm(I_1)$.  Proceeding in this fashion, we obtain a
  sequence of ideals $I_0, I_1, \dots$ in $E$ with
  $$\Nm(I_0)\geq \Nm(I_1)\geq \Nm(I_2)\geq \cdots$$ and
  $$\psi(I_0)>\psi(I_1)>\psi(I_2)>\cdots$$ But the latter
  is clearly impossible, as $\N$ is well-ordered.  We conclude that
  $C$ could not have been Euclidean to begin with.
\end{proof}

With the running restrictions on $K$, we know that any Euclidean ideal
class is represented by a prime of norm $2$ or $3$.  We are led by the
above to examine the union $U$ as above for $C$ a degree one prime
dividing $2$ or $3$.  In determining the extent to which $U$ covers
$\C$, it is clear that one need only consider a particular fundamental
domain for $C$ in $\C$.  Let $K = \Q(\sqrt{-D})$ for a square-free
positive integer $D$.  In each of the following cases, we identify
which $D$ correspond to the case, and for such $D$ we draw a
fundamental domain and the covering circles comprising $U$ that meet
this fundamental domain.  The pictures below are were generated with
SAGE (\cite{sage}).

As we will see, as $D$ increases, the fundamental domains we choose
below get too tall to be covered entirely by these disks.  In each
case, we illustrate the fundamental domain and the disks comprising
$U$ that meet it.  We do this for the following $D$ in each class:
those for which $U$ covers all of $\C$ and the first $D$ for which it
does not (keeping in mind that we are only interested in square-free
$D$).  For the latter $D$, as we will see, it always happens that the
complement of $U$ moreover contains an open set (as opposed to having
a nonempty but discrete complement), so Proposition \ref{prop1}
implies that $C$ is not Euclidean.

We note that in each case below, the given \emph{ideal} generators of $C$ are
also generators of $C$ as an Abelian group, as is easy to check.  Thus
the parallelogram that they span forms the boundary of a
fundamental domain, which is the one that we consider in each case.

\begin{description}
\item[\bf Case 1: $C$ a degree one prime over $2$]\

Since there is a degree one prime dividing $2$, $2$ either ramifies or
splits in $K$, corresponding to the conditions $D\equiv 1,2\pmod{4}$
and $D\equiv 7\pmod{8}$, respectively.  We consider the various
sub-cases separately.

\begin{description}

\item[\bf $2$ ramifies, $D\equiv 1\pmod{4}$]\

Here $\OO_K$ is generated as an algebra over $\Z$ by $\sqrt{-D}$, so
a defining polynomial of $\OO_K$ is $x^2+D$.  Modulo $2$, this is
congruent to $(x+1)^2$, so the unique prime above $(2)$ is
$(2,\sqrt{-D}+1)$.  Thus we use $2$ and $\sqrt{-D}+1$ to span a
parallelogram bounding a fundamental domain, and obtain the following
pictures for increasing $D$.

\begin{center}
  \includegraphics[scale=.26]{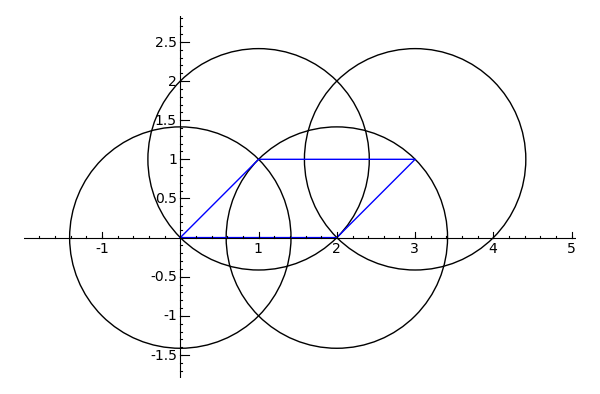}
  \includegraphics[scale=.26]{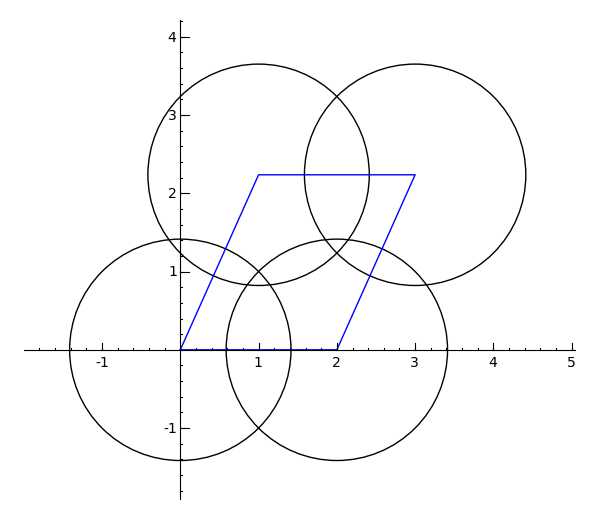}
  \includegraphics[scale=.26]{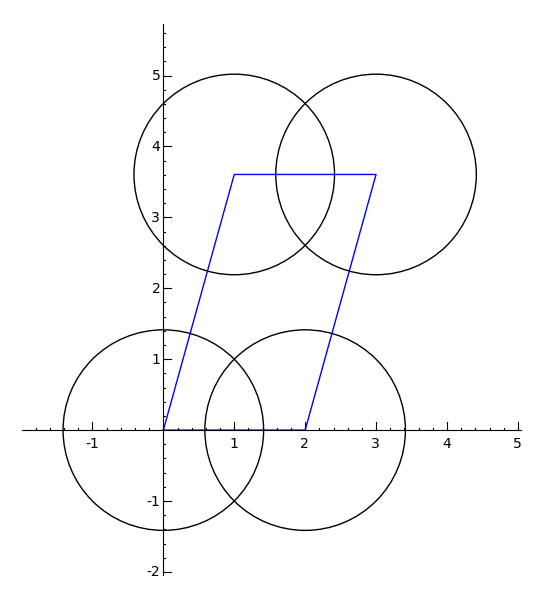}
\end{center}
\begin{center}
  $D=1$ \hspace{1in} $D=5$ \hspace{1in} $D=13$
\end{center}

We conclude that the only $D$ under consideration (recall that $D=1$
was treated separately because of additional units) in this class for
which a degree one prime over $2$ is Euclidean is $D=5$.

\item[\bf $2$ ramifies, $D\equiv 2\pmod{4}$]\

Again the defining polynomial of $\OO_K$ is $x^2+D$.  Modulo $2$, this
is simply $x^2$, so the prime above $(2)$ is $(2,\sqrt{-D})$, and
working as above we obtain the pictures.

\begin{center}
  \includegraphics[scale=.26]{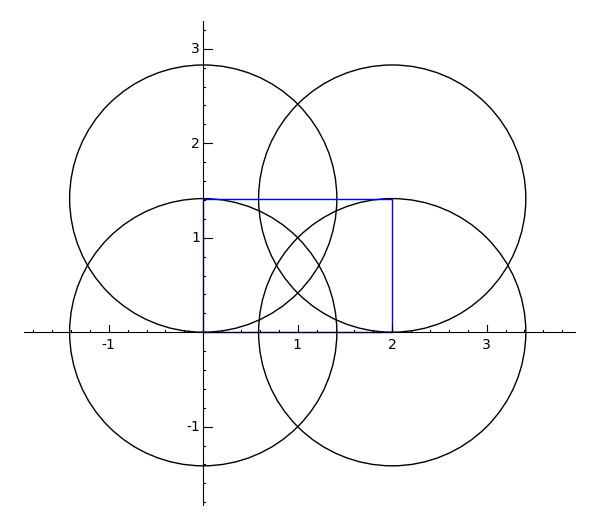}
  \includegraphics[scale=.26]{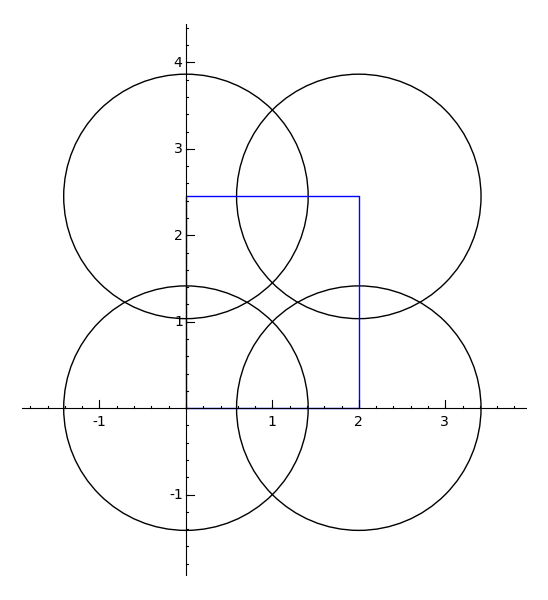}
\end{center}
\begin{center}
  $D=2$ \hspace{1in} $D=6$
\end{center}

We conclude that the only $D$ in this class for which a degree one
prime over $2$ is Euclidean is $D=2$

\item[\bf $2$ splits, $D\equiv 7\pmod{8}$]\

Here $\OO_K$ is generated as an algebra over $\Z$ by
$\frac{1+\sqrt{-D}}{2}$.  The defining polynomial is then
$x^2-x+\frac{1+D}{4}$, which is congruent modulo $2$ to $x(x-1)$.  The
primes above $(2)$ are $(2,\frac{1+\sqrt{-D}}{2})$ and
$(2,\frac{-1+\sqrt{-D}}{2})$.  As these are Galois-conjugate, we need
only examine the first, which gives the following pictures.

\begin{center}
  \includegraphics[scale=.26]{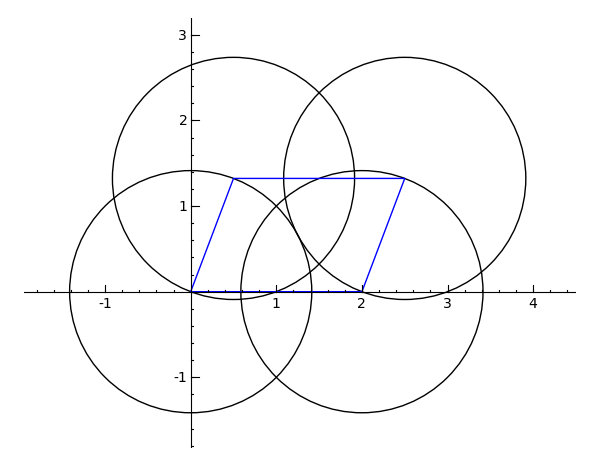}
  \includegraphics[scale=.26]{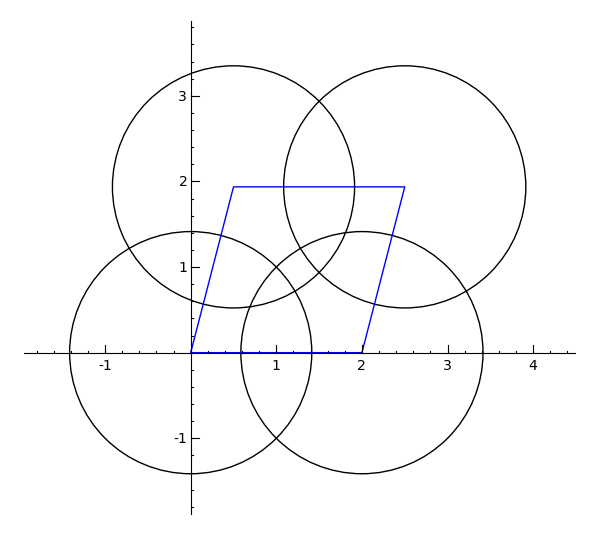}
  \includegraphics[scale=.26]{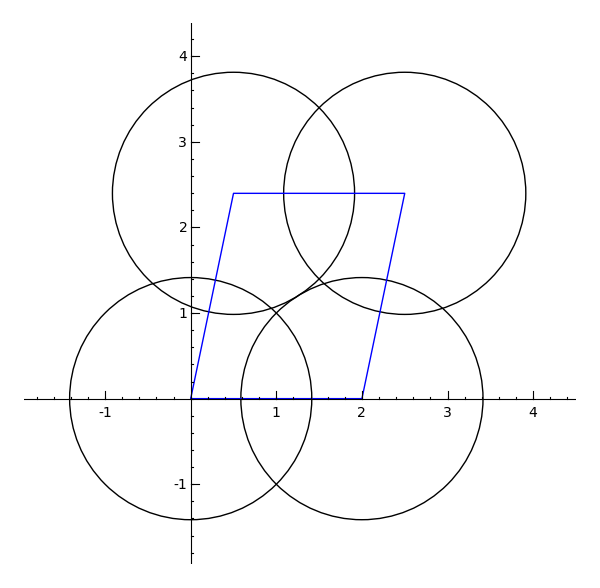}
\end{center}
\begin{center}
  $D=7$ \hspace{1in} $D=15$ \hspace{1in} $D=23$
\end{center}

We conclude that the only $D$ in this class for which a degree one
prime over $2$ is Euclidean are $D=7$ and $D=15$.  It is worth
mentioning that the complement $U$ for $D=23$ is very small.  This
explains the atypical behavior of the sets $A_i$ for $\Q(\sqrt{-23})$
in that they do not stabilize quickly.

\end{description}

\mbox{}\\

\item[\bf Case 2: $C$ is a degree one prime over $3$]\

Next, we examine the case of a degree one primes dividing $3$.  Again,
this means that either $3$ ramifies or splits in $K$, but in order
associate a congruence condition on $D$, we must also take into
account the residue of $D$ mod $4$ since this effects the nature of
the ring of integers $\OO_K$.

\begin{description}

\item[\bf $3$ ramifies, $D\equiv 1,2\pmod{4}$]\

These conditions are equivalent to $D\equiv 6, 9\pmod{12}$.  Here
$\OO_K$ is generated over $\Z$ by $\sqrt{-D}$, so a defining polynomial
is $x^2+D$.  Modulo $3$ this is $x^2$, so the prime above $(3)$ is
$(3,\sqrt{-D})$.  Already for $D=6$, we see that the complement of $U$
contains a nonempty open set.

\begin{center}
  \includegraphics[scale=.26]{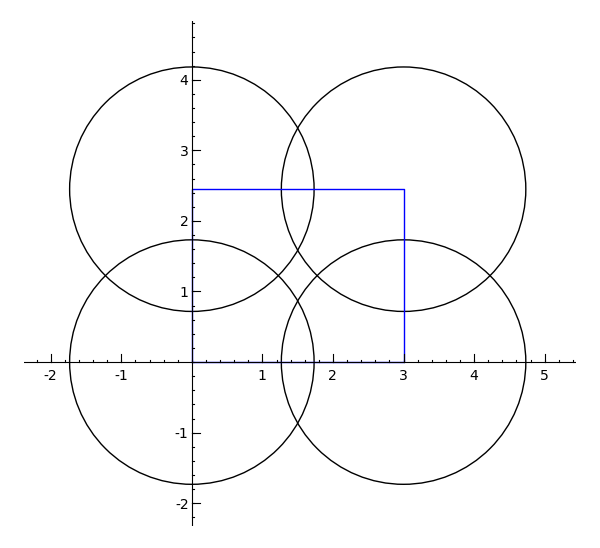}
\end{center}
\begin{center}
  $D=6$ 
\end{center}

We conclude that there are no $D$ in this class for which a degree one
prime over $3$ is Euclidean.

\item[\bf $3$ split, $D\equiv 1,2\pmod{4}$]\

These conditions amount to $D\equiv 2, 5\pmod{12}$.  Again, $x^2+D$ is
a defining polynomial, which is congruent modulo $3$ to $(x-1)(x+1)$.
Thus the primes above $(3)$ are $(3,\sqrt{-D}+1)$ and
$(3,\sqrt{-D}-1)$.  As these are Galois-conjugate, we need only
consider the first, which gives the following pictures.

\begin{center}
  \includegraphics[scale=.26]{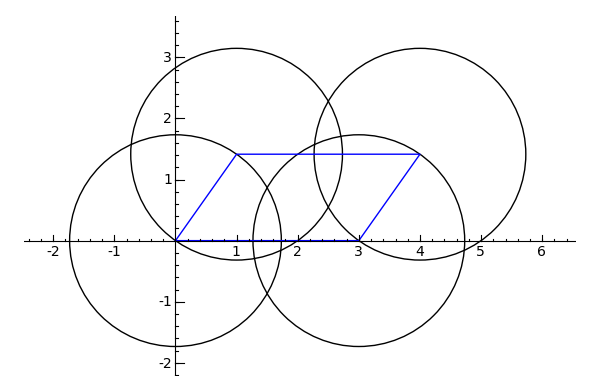}
  \includegraphics[scale=.26]{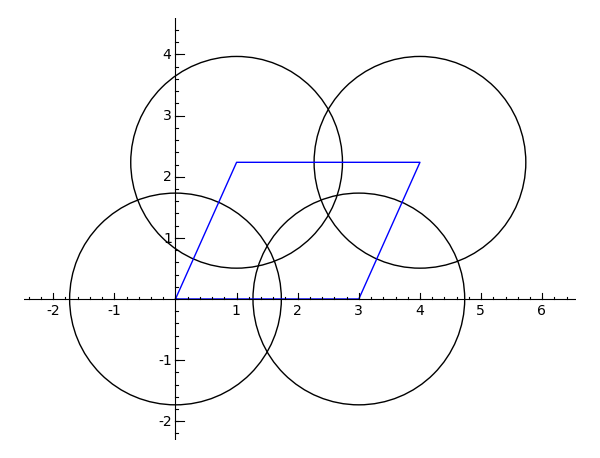}
  \includegraphics[scale=.26]{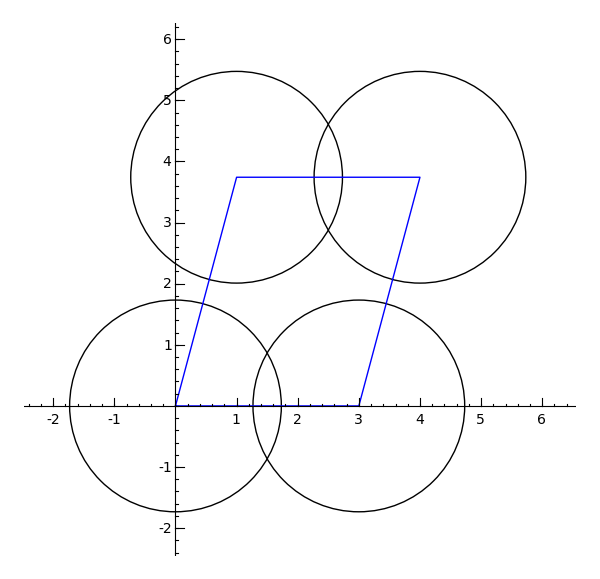}
\end{center}
\begin{center}
  $D=2$ \hspace{1in} $D=5$ \hspace{1in} $D=14$
\end{center}

We conclude that the only $D$ in this class for which a degree one
prime over $3$ is Euclidean are $D=2$ and $D=5$.

\item[\bf $3$ ramifies, $D\equiv 3\pmod{4}$]\

This amounts to $D\equiv 3\pmod{12}$, and in this case
$\frac{1+\sqrt{-D}}{2}$ generates $\OO_K$, and $x^2-x+\frac{1+D}{4}$ is
a defining polynomial.  Modulo $3$, this is congruent to $(x+1)^2$, so
the prime above $(3)$ is $(3,\frac{3+\sqrt{-D}}{2})$.

\begin{center}
  \includegraphics[scale=.26]{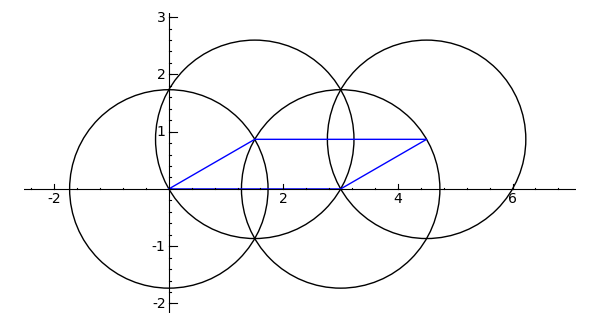}
  \includegraphics[scale=.26]{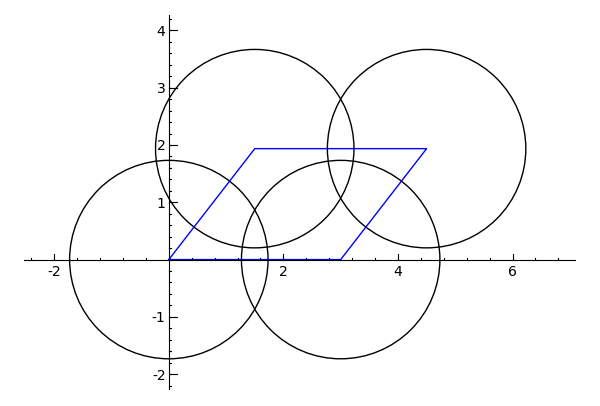}
  \includegraphics[scale=.26]{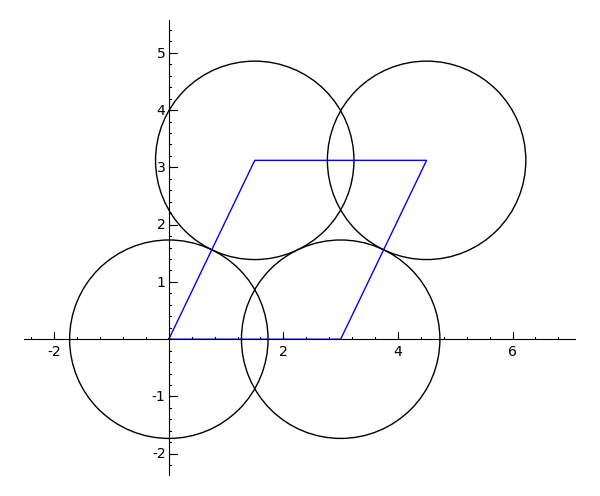}
\end{center}
\begin{center}
  $D=3$ \hspace{1in} $D=15$ \hspace{1in} $D=39$
\end{center}

We conclude that the only $D$ under consideration (recall that $D=3$
was treated separately because of extra units) in this class for which
a degree one prime over $3$ is Euclidean is $D=15$.

\item[\bf $3$ splits, $D\equiv 3\pmod{4}$]\

This amounts to $D\equiv 11 \pmod{12}$.  Again, $\OO_K$ is generated by
$\frac{1+\sqrt{-D}}{2}$ and $x^2-x+\frac{1+D}{4}$ is a defining
polynomial.  Modulo $3$, this is $x(x-1)$, so the primes above $(3)$
are $(3,\frac{1+\sqrt{-D}}{2})$ and $(3,\frac{-1+\sqrt{-D}}{2})$.  As
these are Galois-conjugate, we need only consider the first, which
gives the following pictures.

\begin{center}
  \includegraphics[scale=.26]{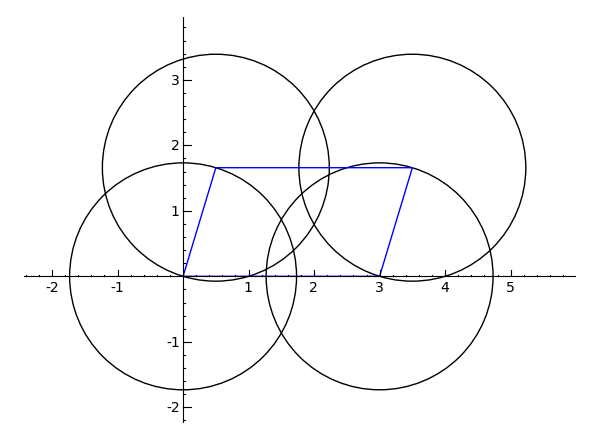}
  \includegraphics[scale=.26]{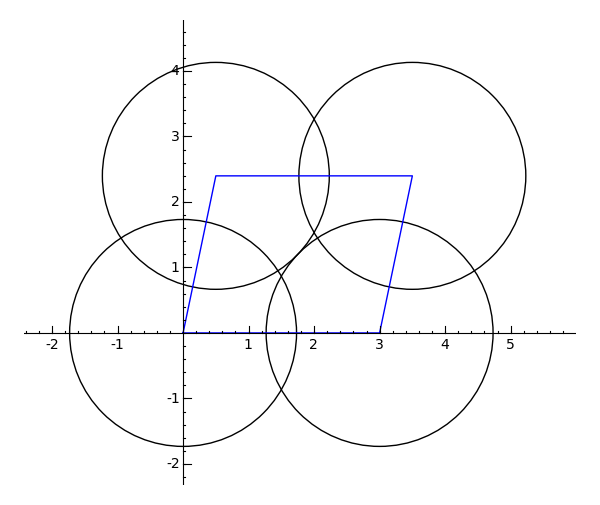}
\end{center}
\begin{center}
  $D=11$ \hspace{1in} $D=23$
\end{center}

We conclude that the only $D$ in this class for which a degree one
prime over $3$ is Euclidean is $D=11$.

\mbox{}\\

\end{description}

\end{description}

The upshot of this enumeration is that the only $D$ for which
$\Q(\sqrt{-D})$ has a Euclidean ideal class are $D\in \{1, 2, 3, 5, 7,
11, 15\}$, and each the unique generator of the class-group is in fact
norm-Euclidean, establishing Theorem \ref{thm1}.

\providecommand{\bysame}{\leavevmode ---\ }
\providecommand{\og}{``}
\providecommand{\fg}{''}
\providecommand{\smfandname}{\&}
\providecommand{\smfedsname}{\'eds.}
\providecommand{\smfedname}{\'ed.}
\providecommand{\smfmastersthesisname}{M\'emoire}
\providecommand{\smfphdthesisname}{Th\`ese}

\bibliographystyle{smfplain}

\end{document}